\documentclass[11pt]{extarticle}
\usepackage[format=plain, font=footnotesize]{caption}
\usepackage{amsmath, amsthm, amssymb, mathtools, hyperref}
\usepackage[shortlabels]{enumitem}
\usepackage{graphicx}
\usepackage{tikz}
\usepackage{xcolor}
\usepackage{listings}
\usepackage{url}
\usetikzlibrary{arrows,calc,fit,matrix,positioning,cd}
\usepackage{cleveref}
\usepackage[all]{xy}
\oddsidemargin \evensidemargin
\definecolor{turqouise}{rgb}{0.07, 0.41, 0.35}
\definecolor{linkblue}{rgb}{0.1, 0.46, 0.82}
\hypersetup{
  colorlinks=true,
  citecolor=turqouise,
  linkcolor=linkblue,
  urlcolor=black
}

\newtheorem{theorem}{Theorem}
\numberwithin{theorem}{section}
\newtheorem{proposition}[theorem]{Proposition}
\newtheorem{lemma}[theorem]{Lemma}
\newtheorem{corollary}[theorem]{Corollary}

\theoremstyle{definition}
\newtheorem{definition}[theorem]{Definition}
\newtheorem{remark}[theorem]{Remark}

\crefname{equation}{}{}
\crefname{equation}{}{}
\crefname{figure}{Figure}{Figure}
\crefname{section}{Section}{Section}
\crefname{lemma}{Lemma}{Lemma}
\crefname{proposition}{Proposition}{Proposition}
\crefname{theorem}{Theorem}{Theorem}
\crefname{corollary}{Corollary}{Corollarie}
\crefname{definition}{Definition}{Definition}
\crefname{notation}{Notations}{Notation}
\crefname{remark}{Remark}{Remark}
\crefname{claim}{Claim}{Claim}
\crefname{assumption}{Assumption}{Assumption}

\newcommand{\R}{\mathbb{R}}
\newcommand{\IR}{\mathbb{IR}}
\newcommand{\C}{\mathbb{C}}
\newcommand{\IC}{\mathbb{IC}}

\title{Certifying zeros of polynomial systems using interval arithmetic}
\date{}
\author{Paul Breiding\thanks{PB: MPI MiS Leipzig. paul.breiding@mis.mpg.de. Funded by the Deutsche Forschungsgemeinschaft (DFG, German Research Foundation) -- Projektnummer 445466444; and funded by the European Research Council
(ERC) under the European Union's Horizon 2020 research and innovation programme (grant agreement No 787840)
},
Kemal Rose\thanks{KR: MPI MiS Leipzig. kemal.rose@mis.mpg.de. Funded by the European Research Council
(ERC) under the European Union's Horizon 2020 research and innovation programme (grant agreement No 787840)},
and
Sascha Timme\thanks{ST: Technische Universit\"at Berlin, Chair of Discrete Mathematics/Geometry. timme@math.tu-berlin.de. Supported by the Deutsche Forschungsgemeinschaft (German Research Foundation) Graduiertenkolleg {\em Facets of Complexity} (GRK~2434)}
}

\begin{document}
\maketitle
\begin{abstract}
We establish interval arithmetic as a practical tool for certification in
numerical algebraic geometry.
Our software \texttt{HomotopyContinuation.jl} now has
a built-in function \texttt{certify}, which proves the correctness of an isolated nonsingular
solution to a square system of polynomial equations.
The implementation rests on Krawczyk's method.
We demonstrate that it dramatically outperforms earlier approaches to certification.
We see this contribution as {powerful new tool} in numerical algebraic geometry, {that can make certification the default} and not just an option.
\end{abstract}

\section{Introduction}
Systems of polynomial equations appear in many areas of mathematics, as well as in many applications in the sciences and engineering.
In physics and chemistry the geometry of molecules is often modelled with algebraic constraints on the distance or the angle between atoms.
In kinematics the relation between robot joints is defined by polynomial equations.
In systems biology the steady-state equations for many bio-chemical reaction networks are algebraic equations.
A central task in all those applications is computing the isolated zeros of a system of polynomials.

The study of zeros of polynomial systems is at the heart of algebraic geometry.
The field of \emph{computational algebraic geometry} is often associated with symbolic computations based on Gr\"obner bases.
But over the last thirty years \emph{numerical algebraic geometry} (NAG) \cite{Sommese:Wampler:2005} emerged as an alternative; enabling us to solve problems infeasible with symbolic methods. The algorithmic framework in NAG is \emph{numerical homotopy continuation}.
Several implementations of this are available: \texttt{Bertini} \cite{Bertini}, \texttt{Hom4PS-3} \cite{Hom4PSArticle}, \texttt{HomotopyContinuation.jl} \cite{HC.jl}, \texttt{NAG4M2} \cite{NumericalAlgebraicGeometryArticle} and \texttt{PHCpack} \cite{PHCpack}.
The first and the third author are the developers of \texttt{HomotopyContinuation.jl}.

Hauenstein and Sottile remark in \cite{Hauenstein:Sottile:2012} that while all of these softwares ``routinely and reliably solve systems of polynomial equations with dozens of variables having thousands of solutions'', they have the shortcoming that ``the output is not certified'' and that ``this restricts their use in some applications, including those in pure mathematics''.
To remedy this, Hauenstein and Sottile developed the software \texttt{alphaCertified} \cite{Hauenstein:Sottile:2012}.
It can rigorously certify that Newton's method, starting at a given numerical approximation, converges quadratically to a true zero by using Smale's $\alpha$-theory \cite{Smale:1986}.
Hauenstein and Sottile's contribution to numerical algebraic geometry was a milestone.
Yet, \texttt{alphaCertified} produces rigorous certificates using expensive rational arithmetic.
This turns the big advantage of numerical computations, namely that they are fast, upside-down and makes certification of large problems prohibitively expensive. Thus, up to this point, the majority of researchers in applied algebraic geometry were kept from using numerical methods, because certification was too expensive and because without certification numerical methods can't be used for proofs.

{We give researchers a new\footnote{We do \emph{not} claim to introduce certification methods based on interval arithmetic in general. This is a well established technique. Our contribution is a new concrete implementation of interval arithmetic methods for numerical algebraic geometry.} {powerful tool in numerical algebraic geometry.} Our implementation is integrated in \texttt{HomotopyContinuation.jl} \cite{HC.jl}, so that in principle we can certify \emph{all} zeros of a system of polynomial equations (see \cref{sec:all zeros} below for more details). With a fast implementation certification becomes the default and is not just an option and enables the extensive use of numerical methods for rigorous proofs. This is underlined by at least 15 research works  \cite{breiding2022line, brysiewicz2021tangent, kohn2021adjoints, boege2021marginal, breiding2022equations, early2021planarity, martyanov2021solving, weinstein2021metric, lindberg2021estimating, sturmfels2021beyond, bender2021yet, agostini2021likelihood, brysiewicz2020nodes, sottile2021galois, sturmfels2021likelihood} that have used our implementation in the last two years.}

\subsection{Contribution}
Our contribution to the field of computational and applied algebraic geometry is an extremely fast and easy-to-use implementation of a certification method. This implementation outperforms \texttt{alphaCertified} by several orders of magnitude.
It makes the certification of solutions often a matter of seconds and not hours or days.
{This leap in performance can turn certification in numerical algebraic into default and not just an option.}

Starting from version 2.1 \texttt{HomotopyContinuation.jl} has a function \texttt{certify}\footnote{The technical documentation is available at\\
\url{https://www.juliahomotopycontinuation.org/HomotopyContinuation.jl/stable/certification}
}.
The function \texttt{certify} takes as input a \emph{square polynomial system} $F$ and a \emph{numerical approximation of a complex zero} $x\in\mathbb C^n$ (or a list of zeros). If the output says ``certified'', then this is a rigorous proof that a solution of $F=0$ is near $x$. If the output says ``not certified'', then this does not necessarily mean that there is no zero near $x$, just that the method couldn't find one.  Figure~\ref{fig1} shows an example of~\texttt{certify}.

We combine interval arithmetic and Krawczyk's method with numerical algebraic geometry to rigorously certify solutions to square systems of polynomial equations.
In technical terms, our implementation returns \emph{strong interval approximate zeros}.
We introduce this notion in Definition~\ref{def_interval_zero} below.
The strong interval approximate zero consists of a box in $\mathbb C^n$, which contains a unique true zero of the polynomial system. If the input is a list of zeros, the routine returns a list of distinct strong interval approximate zeros.
Therefore, our method can be used to \emph{prove} hard lower bounds on the number of zeros of a polynomial system.
Combined with theoretical upper bounds this can constitute rigorous mathematical proofs on the number of zeros of such systems. We explain this in more detail in the next subsection.
In addition, if the given polynomial system is real, we give a certificate whether the certified zero is a real zero ({the approximate zero does not need to be real for this}).
The returned boxes may also be used to verify if a real zero is positive real.
Therefore, our method can also be used to prove lower bounds on the number of real and positive real zeros of a polynomial system.

It is also possible to give a square system of rational functions as input to our implementation. Although this article is mostly formulated in terms of polynomial systems, Krawczyk's method also applies to square systems of rational functions ({in fact, to all analytic functions $\mathbb R^n\to\mathbb R^n$; see Section \ref{sec:krawczyk}}). Consequently, all statements about using our implementation for proofs are equally valid of square systems of rational functions.

\subsection{Certifying all zeros}\label{sec:all zeros}

{Our implementation is integrated in \texttt{HomotopyContinuation.jl} \cite{HC.jl}. This is a software for numerically solving systems of polynomials equations via homotopy continuation. The basic idea is as follows: suppose that $F(x)=(f_1(x),\ldots,f_n(x))$ is a system of polynomials in $n$ variables $x=(x_1,\ldots, x_n)$. To compute the solutions of the systems of equation $F(x)=0$ one takes another system of polynomials $G(x)=(g_1(x),\ldots,g_n(x))$, called \emph{start system}, for which the zeros are simple to compute. Then, $F$ and $G$ are joined with a path in the vector space of systems of polynomials. This path defines a homotopy $H(x,t):\mathbb C^n\times \mathbb C\to \mathbb C$, such that $H(x,1)=G(x)$ and $H(x,0)=F(x)$. The zeros of $G$ are continued towards the zeros of $F$ by solving the ODE initial value problems
$\tfrac{\partial H}{\partial t}  + \tfrac{\partial H}{\partial x}\dot x(t) = 0$, where $x(1)$ ranges over the zeros of $G$. For more details see the textbook \cite{Sommese:Wampler:2005}.}

{In the last paragraph there is nothing special about polynomials. This approach works for any analytic functions $F$ and $G$. However, in the case of systems of polynomial equations we can choose~$G$ such that we compute \emph{all} zeros of $F$. This follows from the \emph{Parameter Continuation Theorem} by Morgan and Sommese~\cite{MS1989}: suppose that $F(x)=F(x;p_0)$ is a point in family of polynomials systems $F(x;p)$ that depends polynomially on parameters $p\in\mathbb C^k$. The Parameter Continuation Theorem says that there exists a proper algebraic subvariety $\Sigma \subset \mathbb C^k$ with the following property. Let $\gamma(t)\colon [0, 1] \rightarrow \mathbb C^k$ with $\gamma(0)=p_0$ be a continuous path and~$H(x,t)$ the corresponding homotopy.
If $\gamma((0, 1])\cap \Sigma = \emptyset$, then as $t \to 0$, the limits of the solution paths $x(t)$ with $H(x(t),t)=0$ include all the isolated
solutions to $F(x; p_0)=0$.
This includes both regular solutions and solutions with multiplicity greater than one. Consequently, every parameter outside~$\Sigma$ provides a suitable start system for $F$.}

{The \emph{Parameter Continuation Theorem} implies the existence of start systems, such that we can compute all zeros of $F$, but it does not tell us how to set up these start systems, nor how to compute their zeros. In fact, different choices of families of parametrized systems lead to different start systems and thus different homotopy methods.
In \texttt{HomotopyContinuation.jl} \cite{HC.jl} one can choose between two well-established strategies for choosing start systems: the so-called totaldegree start system and the polyhedral start system \cite{HuSt95}.}

{Coming back to interval arithmetic we see that the zeros computed by polynomial homotopy continuation can be used as input for certification, so that we can \emph{certify all solutions} of a system of polynomial equations. There is one subtlety, though. Although the Parameter Continuation Theorem asserts that in principle we can find all solutions, since homotopy continuation involves numerical computations we can't rule out the possibiliy that some computations of solutions paths fail. Still, combining certification with homotopy continuation always gives \emph{lower bounds} on the number of zeros. This can be exploited in situations, where we know upper bounds. An example of such a scenario in enumerative geometry is discussed in Section \ref{subsec:3264} below, where we certify 3264 real zeros of a system of polynomials that is known to have at most 3264 complex zeros. Another example from optimization is discussed in \cite[Section 3.3]{EDDMV}. Here, certifying all zeros of a system of polynomial equations helps to rigorously compute the minimal Euclidean distance of a point to an algebraic hypersurface.}

\subsection{Comparison to previous works}
There are other implementations of certification methods using Krawczyk's method and interval arithmetic, e.g., the commercial \texttt{MATLAB} package \texttt{INTLAB} \cite{Rump1999}, the \texttt{Macaulay2} package \texttt{NumericalCertification} \cite{LeeM2}, and the \texttt{Julia} package \texttt{IntervalRootFinding.jl} \cite{BS}. {The theory of Krawczyk's method and interval arithmetic are explained, for instance, in \cite{Rump83}}.

Compared to \texttt{INTLAB}, the source code of our implementation is freely available and can be verified by anyone.
Additionally, \texttt{INTLAB} doesn't support the use of arbitrary precision interval arithmetic which limits its capability to certify badly conditioned solutions.
\texttt{NumericalCertification}, as of version 1.0, takes as input not the numerical approximation of a complex zero $x\in\mathbb C^n$, but instead a box $I$ in $\C^n$. Then, \texttt{NumericalCertification} attempts to certify that interval $I$ is a strong interval approximate.
The process of going from a numerical approximation $x$ to a good candidate interval~$I$ needs particular care, as illustrated in Section \ref{sec:implementation_details}.
\texttt{INTLAB} and \texttt{NumericalCertification} also both require manual work to obtain a list of all distinct distinct strong interval approximate zeros. The package \texttt{IntervalRootFinding.jl} can find all zeros of a multivariate function inside a given box in $\mathbb R^n$, whereas our implementation works in $\C^n$ and additionally certifies reality of zeros; see Section \ref{sec:reality}.

\subsection{Acknowledgements}
We thank Pierre Lairez for a discussion that initiated this project and for several helpful subsequent discussions on the topic. We thank anonymous referees for useful feedback that improved the article.

\subsection{Outline}
{
The rest of this article is organized as follows. In the next two sections we  give a short introduction to interval arithmetic
and explain the Krawczyk method.
Section \ref{section:implementation_details} focuses on implementation details. In Section \ref{sec:experiments} we demonstrate features of our implementations using two examples. For completeness, we give a proof of Krawczyk's method in Section \ref{proof of Krawczyk Method}, and in Section \ref{sec:reality} we discuss how to certify reality of zeros.
}

\section{Interval arithmetic}\label{sec:IA}

{Before we discuss our implementation, let us briefly introduce the basics of interval arithmetic.}

Since the 1950s researchers \cite{Moore:1966,Sunaga:1958} have worked on interval arithmetic, which allows certified computations while still using floating point arithmetic. We briefly introduce the concepts from interval arithmetic which are relevant for our article.

\subsection{Real interval arithmetic}

Real interval arithmetic concerns computing with compact real intervals. Following \cite{Mayer:2017} we denote the set of all compact real intervals by
$$\mathbb I\mathbb R:= \{[a,b]\mid a,b\in\mathbb R, a\leq b\}.$$
For $X, Y \in \IR$ and the binary operation $\circ \in \{ +,-, \cdot, / \}$ we define
\begin{equation}\label{real_interval_algebra}
X \circ Y= \{ x \circ y \,|\, x\in X,y\in Y\},
\end{equation}
where we assume $0 \notin Y$ in the case of division.
The interval arithmetic version of these binary operations, as well as other standard arithmetic operations, have explicit formulas. See, e.g., \cite[Sec. 2.6]{Mayer:2017} for more details.

\subsection{Complex interval arithmetic}
We define the set of \emph{rectangular complex intervals} as
$$\IC :=\{X + i Y\mid X,Y\in\IR\},$$
where $X+iY = \{x+iy\mid x\in X,y\in Y\}$ and $i = \sqrt{-1}$. In particular, every rectangular interval $I\in \IC$ has the property that 
\begin{equation}\label{rectangle_property}
x+iy, a+ib\in I \quad\Longrightarrow \quad x+ib\in I.
\end{equation}
Following \cite[Ch. 9]{Mayer:2017} we define the algebraic operations for $I = X+iY, J=W+iZ\in\IC$ in terms of operations on the real intervals from (\ref{real_interval_algebra}):
\begin{alignat}{2}\label{complex_arithmetic}
I + J &:= (X + W) + i (Y+Z),\qquad I \cdot J &&:= (X\cdot W - Y \cdot Z) + i (X\cdot Z + Y\cdot W)\\\nonumber
I - J &:= (X - W) + i (Y-Z),\qquad \;\, \frac{I}{J} &&:= \frac{X \cdot W + Y\cdot Z}{W\cdot W + Z\cdot Z}  + i \frac{Y \cdot W - X\cdot Z}{W\cdot W + Z\cdot Z}.
\end{alignat}
It is necessary to use (\ref{real_interval_algebra}) instead of complex arithmetic for the definition of algebraic operations in $\IC$.
The following example from \cite{Mayer:2017} demonstrates this. Consider the intervals $I=[1,2] + i[0,0]$ and~$J = [1,1] + i[1,1]$.
Then, $\{x\cdot y| x\in I,y\in J\} = \{t(1+i)\mid 1\leq t\leq 2\}$ is not a rectangular complex interval, while $I\cdot J = [1,2] + i [1,2]$ is.

The algebraic structure of $\IC$ is given by following theorem; see, e.g., \cite[Theorem 9.1.4]{Mayer:2017}.
\begin{theorem}\label{thm_arithmetic}
The following holds.
\begin{enumerate}
\item $(\IC, +)$ is a commutative semigroup with neutral element.
\item $(\IC, +, \cdot)$ has no zero divisors.
\end{enumerate}
Furthermore, if $I,J,K,L\in\IC$, then
\begin{enumerate}[resume]
\item $I\cdot (J + K) \subseteq I\cdot J + I\cdot K$, but equality does not hold in general.
\item $I\subseteq J, K\subseteq L$, then $I \circ K\subseteq J\circ L$ for $\circ\in\{+,-,\cdot, /\}$.
\end{enumerate}
\end{theorem}
Working with interval arithmetic is challenging because of the third item from the previous theorem: distributivity does not hold in $\IC$.
As a consequence, in $\IC$ the evaluation of polynomials depends on the exact order of the evaluation steps.
Therefore, the evaluation of polynomial maps $F: \IC^n \to \IC$ is only well-defined if $F$ is defined by a straight-line program, and not just by a list of coefficients. Figure \ref{fig2} demonstrates this issue in an example. See, e.g., \cite[Sec. 4.1]{burgisser2013algebraic} for an introduction to straight-line programs.

\begin{figure}[hb]
 \begin{minipage}{0.49 \textwidth}
     \centering
     \begin{tikzpicture}
         \node[shape=circle,draw=black, minimum size=6mm,inner sep=0pt] (x) at (0,3) {x};
         \node[shape=circle,draw=black, minimum size=6mm,inner sep=0pt] (y) at (1,3) {y};
         \node[shape=circle,draw=black, minimum size=6mm,inner sep=0pt] (z) at (2,3) {z};
         \node[shape=circle,draw=black, minimum size=6mm,inner sep=0pt] (+) at (0.5, 2) {+};
         \node[shape=circle,draw=black, minimum size=6mm,inner sep=0pt] (*) at (1.5,1) {$\cdot$};
         \path [->](x) edge (+);
         \path [->](y) edge (+);
         \path [->](+) edge (*);
         \path [->](z) edge (*);
 \end{tikzpicture}
 \end{minipage}
 \begin{minipage}{0.49 \textwidth}
     \centering
     \begin{tikzpicture}
         \node[shape=circle,draw=black, minimum size=6mm,inner sep=0pt] (x) at (0,3) {x};
         \node[shape=circle,draw=black, minimum size=6mm,inner sep=0pt] (z) at (1,3) {z};
         \node[shape=circle,draw=black, minimum size=6mm,inner sep=0pt] (y) at (2,3) {y};
         \node[shape=circle,draw=black, minimum size=6mm,inner sep=0pt] (*1) at (0.5, 2) {$\cdot$};
         \node[shape=circle,draw=black, minimum size=6mm,inner sep=0pt] (*2) at (1.5, 2) {$\cdot$};
         \node[shape=circle,draw=black, minimum size=6mm,inner sep=0pt] (+) at (1,1) {+};
         \path [->](x) edge (*1);
         \path [->](y) edge (*2);
         \path [->](z) edge (*1);
         \path [->](z) edge (*2);
         \path [->](*1) edge (+);
         \path [->](*2) edge (+);
 \end{tikzpicture}
 \end{minipage}
\caption{\label{fig2}The picture shows two straight-line programs for evaluating the polynomial $f(x,y,z)=(x+y)z$.  Let $I=([-1, 0],\,[1, 1],\, [0, 1])^T$. Then, the program on the left evaluated at $I$ yields $f(I) = ( [-1, 0]  + [1, 1] ) [0, 1] = [0,1]$, while the program on the right yields $f(I) = [-1, 0]  [0, 1]+ [1, 1][0, 1]   = [-1,1]$.}
\end{figure}

Arithmetic in $\IC^n$ is defined in the expected way. If $I=(I_1,\ldots,I_n),J=(J_1,\ldots,J_n) \in\IC^n$,
$$I + J = (I_1+J_1,\ldots,I_n+J_n).$$
Scalar multiplication for $I\in\IC$ and $J\in\IC^n$ is defined as $I \cdot J = (I\cdot J_1,\ldots,I\cdot J_n)$.
The product of an interval matrix $A=(A_{i,j})\in\IC^{n\times n}$ and an interval vector $I\in\IC^n$ is
\begin{equation}\label{def_matrix_mult}A\cdot I := I_1\cdot \begin{bmatrix}
A_{1,1}  \\
\vdots\\
A_{n,1}
 \end{bmatrix} + \cdots +
 I_n\cdot \begin{bmatrix}
A_{1,n}  \\
\vdots\\
A_{n,n}
 \end{bmatrix}.
\end{equation}
Similar to the one-dimensional case $(\IC^n, +)$ is a commutative semigroup with neutral element.

\section{Certifying zeros with interval arithmetic}\label{sec:krawczyk}

In 1969 Krawczyk \cite{Krawczyk:1969} developed an interval arithmetic version of Newton's method. Later in 1977 Moore \cite{Moore:1977} recognized that Krawczyk's method can be used to certify the existence and uniqueness of a solution to a system of nonlinear equations.
Interval arithmetic and interval Newton's method are a prominent tool in many areas of applied mathematics; e.g., in chemical engineering \cite{Gopalan:1995}, thermodynamics \cite{Hatice:2005} and robotics \cite{Kumar:Sen:Shooe:2015}. {See also the overview in~\cite{Rump2010}.}

{The results in this section are stated for general functions.
For a practical implementation it is however necessary to compute interval enclosures (see \cref{enclosure}). We discuss our approach in the context of polynomial systems in Section \ref{sec:IE} below.
A generalization in this spirit is discussed in \cite{BLL:2019}.}

\subsection{Krawczyk's method}
In this section we recall Krawczyk's method. First, we need three definitions.

\begin{definition}[Interval enclosure]\label{enclosure}
Let $F: \mathbb{C}^n \rightarrow \mathbb{C}^{n}$. A map $\square F: \mathbb{IC}^n \rightarrow \mathbb{IC}^{n}$ is an interval enclosure of $F$ if for every $I \in \mathbb{IC}^n$ we have
$
\{F(x) \mid x \in I\} \subseteq \square F(I).$
\end{definition}
In the rest of this article we use the notation $\square F$ to denote the interval enclosure of $F$.
Also, we do not distinguish between a point $x\in\mathbb C^{n}$ and the complex interval $[\mathrm{Re}(x),\mathrm{Re}(x)] + i[\mathrm{Im}(x),\mathrm{Im}(x)]$ defined by $x$. We simply use the symbol ``$x$'' for both terms so that $\square F(x)$ is well-defined.
\begin{definition}[Interval matrix norm]\label{matrix_norm}
Let $A\in\mathbb{IC}^{n\times n}$. We define the operator norm of $A$ as
$\Vert A\Vert_\infty := \max\limits_{B\in A} \max\limits_{v\in\mathbb C^n} \tfrac{\Vert Bv\Vert_\infty}{\Vert v\Vert_\infty},$
where $\Vert (v_1,\ldots,v_n) \Vert_\infty = \max_{1\leq i\leq n}\vert v_i\vert$ is the infinity norm in $\mathbb C^n$.
\end{definition}

Next we introduce an interval version of the Newton operator, the \emph{Krawczyk operator} \cite{Krawczyk:1969}.
\begin{definition}\label{def_krawczyk}
Let $F: \mathbb{C}^n \rightarrow \mathbb{C}^n$ be differentiable, and $\mathrm{J}F$ be its Jacobian matrix seen as a function $\mathbb C^n \to \mathbb C^{n\times n}$.
Let $\square F$ be an interval enclosure of $F$ and $\square \mathrm{J}F$ be an interval enclosure of~$\mathrm{J}F$.
Furthermore, let $I \in \mathbb{IC}^n$ and $x \in \mathbb C^n$ and let $Y \in \mathbb C^{n\times n}$ be an invertible matrix.
We define the Krawczyk operator
\begin{align*}
K_{x,Y}(I) := x - Y \cdot \square F(x) + (\mathbf 1_n - Y \cdot \square \mathrm{J}F(I))(I-x).
\end{align*}
Here, $\mathbf 1_n$ is the ${n\times n}$-identity matrix.
\end{definition}

\begin{remark}
In the literature, $K_{x,Y}(I)$ is often defined using $F(x)$ and not $\square F(x)$.
Here, we use this definition, because in practice it is usually not feasible to evaluate $F(x)$ exactly.
Instead, $F(x)$ is replaced by an interval enclosure.
\end{remark}
\begin{remark}
The second part of Theorem \ref{thm:krawczyk} motivates to find a matrix $Y \in \C^{n \times n}$ such that $|| 1_n - Y \cdot \square \mathrm{J}F(I) ||_\infty$ is minimized.
 A good choice is an approximation of the inverse of $\mathrm JF(x)$.
\end{remark}

We are now ready to state the theorem behind Krawczyk's method.
The first proof for real interval arithmetic is due to Moore \cite{Moore:1977}. {A proof for complex data is at least known since the work by Rump \cite{Rump83}.
We recall the proof from \cite{BLL:2019} in Section \ref{proof of Krawczyk Method} below.
Note that all the data in the theorem can be computed using interval arithmetic.}

\begin{theorem}\label{thm:krawczyk}
Let $F: \mathbb{C}^n \rightarrow \mathbb{C}^n$ be differentiable and $I\in\IC^n$.
Let $x\in I$ and $Y\in\mathbb{C}^{n\times n}$ be an invertible complex $n\times n$ matrix.
The following holds.
  \begin{enumerate}
    \item If $K_{x,Y}(I) \subset I$, there is a zero of $F$ in $I$.
    \item If additionally
    $\sqrt{2} \, \lVert \mathbf 1_n - Y \square \mathrm{J}F(I) \rVert_\infty < 1$, then
    $F$ has exactly one zero in $I$.
    \end{enumerate}
\end{theorem}

To simplify our language when talking about intervals $I \in \IC^n$ satisfying Theorem \ref{thm:krawczyk} we introduce the following definitions.
\begin{definition}\label{def_interval_zero}
  Let $F: \C^n \rightarrow \C^n$ be differentiable and $I \in \IC^n$. Let $K_{x,Y}(I)$ be the associated Krawczyk operator (see Definition \ref{def_krawczyk}).
  If there exists an invertible matrix $Y\in\mathbb C^{n\times n}$, such that $K_{x,Y}(I)\subset I$, we say that $I$ is an \emph{interval approximate zero}~$F$.
  We call $I$ a \emph{strong interval approximate zero} of $F$ if in addition $\sqrt{2} \lVert \mathbf 1_n - Y \square \mathrm{J}F(I) \rVert_\infty < 1$ .
\end{definition}

\begin{remark}
{The name ``strong interval approximate zero'' is not common in the field of interval arithmetic. We introduce it as a reference to the work of Shub and Smale and the software \texttt{alphaCertified} \cite{Hauenstein:Sottile:2012} that inspired our work. Shub and Smale coined the name \emph{strong approximate zero} for points in the radius of quadratic convergence of Newton's method.}
\end{remark}
\begin{definition}
  If $I$ is an interval approximate zero, then, by Theorem \ref{thm:krawczyk}, $I$ contains a zero of~$F$. We call such a zero an \emph{associated zero} of $I$.
  If $I$ is a strong interval approximate zero then there is a unique associated zero and we refer to is as \emph{the} associated zero of $I$.
\end{definition}

The notion of strong interval approximate zero is stronger than the definition suggests at first sight. We not only certify that a unique zero of $F$ exists inside $I$, but even that we can approximate this zero with arbitrary precision. This is shown in the next proposition, which we prove in Section~\ref{proof of Krawczyk Method}
\begin{proposition}\label{interesting_prop}
Let $I$ be a strong interval approximate zero of $F$ and let $x^*\in I$ be the unique zero of $F$ inside $I$. Let $x\in I$ be any point in $I$. We define $x_0:=x$ and for all $i\geq 1$ we define the iterates $x_i := x_{i-1} - Y\,F(x_{i-1})$, where $Y\in\mathbb C^{n\times n}$ is the matrix from Definition \ref{def_interval_zero}. Then, the sequence $(x_i)_{i\geq 0}$ converges (at least linearly) to $x^*$.
\end{proposition}

\section{Implementation details}\label{sec:implementation_details}
\label{section:implementation_details}

In this section we describe details of our implementation of Krawcyzk's method.

The certification routine takes as input a square polynomial system $F: \C^n \rightarrow \C^n$ and a finite list $X \subset \C^n$ of  (suspected) approximations of isolated nonsingular zeros of $F$.
It is also possible to provide a square system of rational functions as input, but in the following we focus in polynomial systems for simplicity.
Our implementation returns a list of strong interval approximate zeros $\mathcal{I}=\{I_1,\ldots,I_m\}$ in $\IC^n$, such that no two intervals $I_k$ and $I_\ell$, $k \ne \ell$, overlap.
If two strong interval approximate zeros don't overlap then this implies that their associated zeros are distinct.
Additionally, if $F$ is a real polynomial system then for each $I_k \in \mathcal{I}$ it is determined whether its associated zero is real.
The prototypical application of the certification routine is to take as input approximations of all isolated nonsingular solutions $X \subset \C^n$ of $F$ as computed by numerical homotopy continuation methods as discussed in Section \ref{sec:all zeros}.

\subsection{Interval enclosures for polynomial systems}\label{sec:IE}
The fact that distributivity doesn't hold in $\IC$ makes it necessary for us to define the polynomial system $F: \C^n \to \C^n$, and its interval enclosure $\square F$, by a straight-line program, and not just by a list of coefficients.
The overestimation of the interval enclosure $\square F$ increases with the size of the straight line program.
Therefore, it is good to express $F$ and its enclosure $\square F$ by the smallest straight line program possible.
To achieve this, \texttt{HomotopyContinuation.jl} automatically applies
a multivariate version of Horner's rule to reduce the number of operations necessary to evaluate $F$ and $\square F$.

\begin{remark}
Our implementation of interval enclosures can also be used to prove that a polynomial map $F:\mathbb C^n\to\mathbb C^m$ with real coefficients, evaluated at a real point $p\in\mathbb R^n$, is positive. To verify this, one takes an interval $I\in \IC^n$ of the form $I=J + i[0,0]^{\times n}$ such that $p\in J$. If $\square F$ is an interval enclosure of $F$, and if $\square F(I)\subset \mathbb R_{>0}^m + i[0,0]^{\times m}$, then this is a proof that $F(p) \in \mathbb R_{>0}^m$.
\end{remark}

\subsection{Machine interval arithmetic}
In the next subsection we give a method to construct an candidate $I \in \IC^n$ for a strong interval approximate zero.
Before, we need to study \emph{machine interval arithmetic}; the realization of interval arithmetic with finite precision floating point arithmetic.
We assume the standard model of floating point arithmetic \cite[Section 2.3]{Higham:2002}, where the result of a floating point operation is accurate up to relative unit roundoff $u$: $\mathrm{fl}(x \circ y) = (x  \circ  y)(1 + \delta),$ where $|\delta| \le u$ and $\circ\in\{+,-,*,/\}$.
For instance, following the IEE-754 standard, the unit roundoff in double precision arithmetic is $u = 2^{-53} \approx 2.2 \cdot 10^{-16}$.
The key property in the context of interval arithmetic is that each result of a floating point operation can be rounded outwards, such that the resulting \emph{interval} contains the true (exact) result; see, e.g., \cite[Section 3.2]{Mayer:2017}.
Therefore, given $X, Y \in \IC$ the result of $X \circ Y$, $\circ\in\{+,-,*,/\}$, is $\mathrm{fl}(X \circ Y) := \{ (x  \circ  y)(1+\delta) \,|\, |\delta| \le u, x\in X,y\in Y \}$ in machine arithmetic. This interval contains $X\circ Y$. It is \emph{larger}.
Additionally, for a given $x \in \IC$, all intervals of the form
$\{x + (|\mathrm{Re}(x_j)| + i |\mathrm{Im}(x_j)|)\delta \, | \, |\delta| \le \mu \}$ with $0 < \mu \le u$ are indistinguishable when working with precision $u$.

Consequently it is possible that the Krawczyk operator $K_{\tilde{x},Y}$, see Definition \ref{def_krawczyk},
is a contraction for the interval $I$, but that machine arithmetic can't verify this, because $\mathrm{fl}(X  \circ  Y)$ is larger than $X\circ Y$.
In such a case, the unit roundoff $u$ needs to be sufficiently decreased.
For this reason our implementation uses machine interval arithmetic based on double precision arithmetic as well as, if necessary, the arbitrary precision interval arithmetic implemented in \texttt{Arb} \cite{Johansson2017arb}.
{For instance, we could not certify all solutions in the example in Section \ref{subsec:3264} below using only 64-bit arithmetic, because the zeros are too ill-posed.}

\subsection{Determining strong interval approximate zeros}\label{sec:inflation}

In a first step, the certification routine attempts to produce for a given $x \in X$ a strong interval approximate zero $I \in \IC^n$.
Recall that for $I \in \IC^n$ to be a strong interval approximate zero we need by Theorem \ref{thm:krawczyk} to have a point $\tilde{x} \in I$, and a matrix $Y \in \C^{n \times n}$ such that $K_{\tilde{x},Y}(I) \subset I$, and~$\sqrt{2} \, \lVert \mathbf 1_n - Y \square \mathrm{J}F(I) \rVert_\infty < 1$.

Given a point $x \in X$ and a unit roundoff $u$, the point $x$ is refined  using Newton's method to maximal accuracy. We denote this refined point $\tilde{x}$.
Here, we assume that $x$ is already in the region of quadratic convergence of Newton's method.
Next, the point $\tilde{x}$ needs to be inflated to an interval~$I$ with $\tilde{x} \in I$.
This process is called $\varepsilon$-inflation in the literature \cite[Sec. 4.3]{Mayer:2017}.
However, choosing the correct $I$ is a hard problem:
if $I$ is too small or too large, then the Krawcyzk operator is not a contraction.

In spite of these difficulties, we found the following heuristic to determine $I$ work very well. {First, we compute $JF(\tilde{x})^{-1}$ in floating arithmetic, which yields a matrix $Y$. Then, we set $$I := (\tilde{x}_j \pm |(Y \cdot \square F(x))_j|u^{- \frac14})_{j=1,\ldots,n},$$ where $u$ is the unit roundoff. The motivation behind this choice is as follows:
If we assume $\tilde{x}$ to be in the region of quadratic convergence of Newton's method, it follows from the Newton-Kantorovich theorem that $||\mathrm JF(\tilde{x})^{-1}F(\tilde{x})||_\infty$ is a good estimate of the distance between~$\tilde{x}$ and the convergence limit $x^*$. This distance is approximated by $(Y \cdot \square F(x))_j$ for $1\leq j\leq n$. The factor $u^{- \frac14}$ accounts for the overestimation by machine interval arithmetic. Here is how we arrived at this factor: The best relative accuracy we can expect to get for the $j$-th entry of $x^*$ is about $\vert (x^*)_j\vert \cdot u$, so that $|| I - x^* ||_\infty$ needs to be larger than $(|(x^*)_j| u)^{-1/2}$ for quadratic convergence.
On the other hand, we need to have an $\varepsilon$-inflation of at least  $|(Y \cdot \square F({x}))_j|$ so that the inflated interval contains $(x^*)_j$. In the typical case we have $|(Y \cdot \square F({x}))_j|  > 1$, i.e., $|(Y \cdot \square F({x}))_j| > |(Y \cdot \square F({x}))_j|^\frac{1}{2}$. All of this motivates us to use $|(Y \cdot \square F({x}))_j|u^{-\frac{1}{2}}$ as the inflation constant. However, to account for hidden constant factors we need to increase this estimate. We found that replacing $u^{-\frac{1}{2}}$ by $u^{-\frac{1}{4}}$ produces a good estimate  that works well in all the examples we tested.}

Finally, if $I$ doesn't satisfy the conditions in Theorem~\ref{thm:krawczyk} the procedure is repeated with a smaller unit roundoff $u$.
This repeats until either a minimal unit roundoff is reached or the certification is successful.

\subsection{Producing distinct intervals}\label{sec_disjoint}
Assume now that the steps in Section \ref{sec:inflation} have been performed for all $x \in X$. We obtain a list of strong interval approximate zeros $I_1,\ldots,I_r \in \IC^n$.
In a final step we want to select a subset $M \subset \{1,\ldots,r\}$ such that for all $k, j \in M$, $k \ne j$, the intervals $I_k$ and $I_j$ do not overlap.
If two strong interval approximate zeros do not overlap then it is guaranteed that they have distinct associated zeros.
A simple approach to determine $M$ is to compare all intervals pairwise. However, this approach requires us to perform ${r \choose 2}$ interval vector comparisons. {
For larger problems this becomes prohibitively expensive:
in the example in Section \ref{six-bar-linkages} the number of necessary comparisons is already larger than $4\cdot 10^{9}$.}

Instead, we employ the following improved scheme to determine all non-overlapping intervals.
First, we pick a random point $q \in \C^n$ and compute in interval arithmetic for each $I_k$, $k \in M$, the squared Euclidean distance $d_k \in \IR$ between $I_k$ and $q$.
Due to the guarantees of interval arithmetic we have that $d_k$ and $d_\ell$ overlap if $I_k$ and $I_\ell$ overlap (but the converse it not necessarily true).
Next, we check for all overlapping intervals $d_k, d_\ell\in \{ d_k \in \IR \, | \, k = 1,\ldots,r\}$, whether $I_k$ and $I_\ell$ overlap, and if so, we group them accordingly.
This allows us to construct the set $M$ by selecting those intervals which don't overlap with any other and by picking one representative of each cluster of overlapping intervals.
The worst case complexity of this procedure still requires $O(r^2)$ operations, but in the common case where no or only a small number of intervals overlap $O(r \log r)$ operations are sufficient.

\section{Applications}\label{sec:experiments}
{In this section we showcase two example applications of our certification method.}

{The first example is from enumerative geometry and demonstrates how our method can be used for rigorous proofs. The second example is an application from kinematics, which shows that our implementation can deal with large problems and that our strategy for producing distinct intervals from Section \ref{sec_disjoint} is indispensable. This is underlined by the fact that with our computation we improve a result from the literature.
Both examples emphasize the speed compared to the symbolic approaches, and they rely on the option to modulate the precision thanks to
our usage of \texttt{Arb}~\cite{Johansson2017arb}.}

All reported timings were obtained on an desktop computer with a 3.4 GHz processor running \texttt{Julia}~1.5.2 \cite{BEKV} and \texttt{HomotopyContinuation.jl} version 2.2.2.

\subsection{3264 real conics}\label{subsec:3264}
We demonstrate how certification methods in numerical algebraic geometry allow to proof theorems in algebraic geometry.
This example furthermore reveals the superior speed of our implementation compared to  \texttt{alphaCertified}.

In \cite{3264} we used \texttt{alphaCertified} to prove that a certain arrangement of five conics in the plane had 3264 real conics, which were simultaneously tangent to each of the five given conics.
Such an arrangement is called \emph{totally real}.
It was known before that such arrangements exist \cite{RTV97}, but an explicit instance was not known.
The fact that \texttt{alphaCertified} provides a proof for a totally real instance highlights the relevance of certification software in algebraic geometry.

The strategy for the computation is this.
The zeros of the system (12) in \cite{3264} give the coordinates of the 3264 conics which are tangent to five given conics.
We compute the zeros for the coordinates of the specific instance in \cite[Figure 2]{3264} using \texttt{HomotopyContinuation.jl}.
This is a numerical computation.
Therefore, it is inexact and cannot be used in a proof. Next, we take the inexact numerical zeros as starting points for our certification method.
If our implementation outputs that it has found a real certified zero, then this is an exact result and hence it is a proof that the zero is real.
This way we can prove that indeed all the 3264 conics for the instance in \cite[Figure 2]{3264} are real.
See also the proof of \cite[Proposition 1]{3264} for a more detailed discussion.

The certification with \texttt{alphaCertified} took us \emph{more than 36 hours}.
In contrast, our implementation certifies the reality and distinctness of the 3264 conics in \emph{less than three seconds}.

\begin{figure}
\begin{center}
\includegraphics[width = 0.7\textwidth]{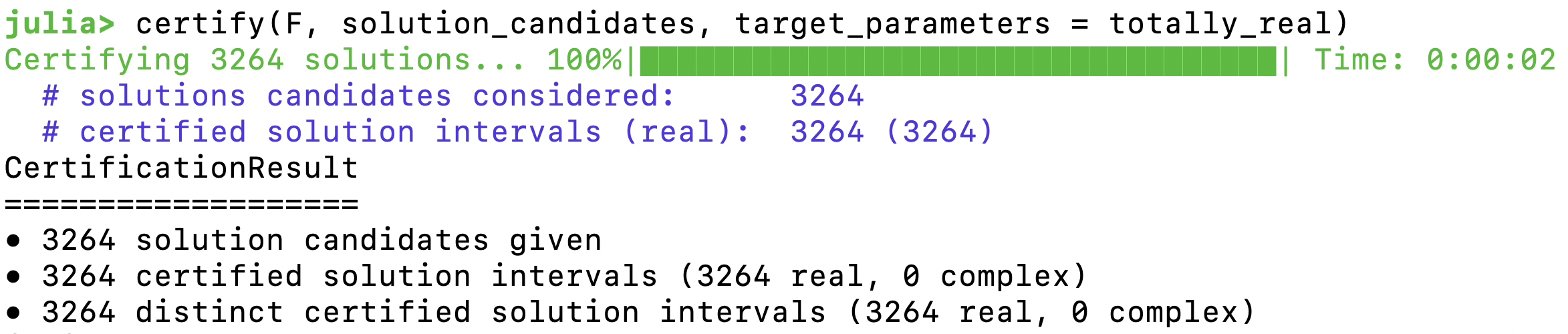}
\end{center}
\caption{\label{fig1}
Screenshot from a \texttt{Julia} session, where we certify the 3264 real conics for the totally real arrangement from \cite{3264}. Here, \texttt{F} is the system of polynomials from (12) in \cite{3264}. The screenshot also demonstrates the simple syntax of our implementation.}
\end{figure}

\subsection{Numerical Synthesis of Six-Bar Linkages}
\label{six-bar-linkages}
Now we demonstrate that the certification routine can cope with large problems. With our computation we improve a result from the literature.

We consider the kinematic synthesis of six-bar linkages that use eight prescribed accuracy points as described in \cite{Plecnik:McCarthy:2014}.
In this article, the authors derive the synthesis equations for six-bar linkages of the Watt II, Stephenson II, and Stephenson III type.
Additionally, in \cite[Eq. (35)]{Plecnik:McCarthy:2014} they construct a system of 22 polynomials in 22 unknowns and 224 parameters, that can be used as a start system in a parameter homotopy to solve the synthesis equations of all three considered six-bar linkage types.

The number of non-singular zeros of this generalized start system is reported as 92,736.
It was computed using \texttt{Bertini} and a multi-homogeneous start system.
To certify the reported count, we solved the generalized start system using the monodromy method \cite{monodromy} implementation in \texttt{HomotopyContinuation.jl}.
In our computation we obtained 92,752 non-singular zeros for a generic choice of the 224 parameters. These are sixteen \emph{more} than reported in \cite{Plecnik:McCarthy:2014}.
We certified this count using our certification routine and obtained 92,752 distinct strong interval approximate zeros.
Therefore, we have a certificate that the generalized system has in general (at least) 92,752 non-singular solution.
This establishes that the result in \cite{Plecnik:McCarthy:2014} undercounts the true number of solutions.
The certification needed only 38.34 seconds which underlines the scalability of the certification routine. {Notice that the naive method for comparing intervals in Section \ref{sec_disjoint} gives $4.301.420.376$ pairs to check. This underlines the need for having an efficient algorithm for comparing pairs.}

\section{Proof of the Krawczyk method}
\label{proof of Krawczyk Method}

The idea for the proof of both Theorem \ref{thm:krawczyk} and Proposition \ref{interesting_prop} is to verify that for strong interval approximate zeros $I$ the map $G_Y(x)= x-Y\cdot F(x)$ defines a contraction on $I$.
If this is true, by Banach's Fixed Point Theorem there is exactly one fixed-point of this map in $I$.
Since $Y$ is invertible, this implies that there is exactly one zero to $F(x)$ in $I$.

Before we give the proof of Theorem \ref{thm:krawczyk}, we need a lemma.
It is a consequence of a complex version of the mean-value theorem.

\begin{lemma}
  \label{lemma: oracle_ext}
Let $Y\in\mathbb C^{n\times n}$ and $F:\mathbb C^n\to\mathbb C^n$. Define $G_Y(x)= x-YF(x)$. Let $I\in\IC^n$ be an interval vector and $x,z\in I$. Then, we have
\begin{enumerate}
\item $\Vert G_Y(z)-G_Y(x) \Vert_\infty \leq \sqrt{2}\,\Vert \mathbf 1_n - Y \cdot\square \mathrm{J}F(I)\Vert_\infty\, \Vert z-x\Vert_\infty.$
\item $G_Y(I)\subset K_{x,Y}(I)$.
\end{enumerate}
 \end{lemma}
\begin{proof}
The following proof is adapted from \cite[Lemma 2]{BLL:2019}.

In the proof we abbreviate $G:=G_Y$. We fix an interval $I\in\IC^n$ and $x,z\in\mathbb C^n$.

The first part of the lemma may be shown entry-wise. We will show this by combining a variant of the mean value theorem with the following observation:  $\mathrm{J}G(x) = \mathbf 1_n - Y \cdot \mathrm{J}F(x)$, so we get, using~\cref{rectangle_property}, the inclusion
  \begin{equation}\label{eq11}
    \mathrm{Re}(\mathrm{J}G(I)) + i \, \mathrm{Im}(\mathrm{J}G(I))  \subseteq\mathbf 1_n - Y \cdot\square \mathrm{J}F(I).
  \end{equation}
We define $w:=\mathrm{Re}(z) + i\mathrm{Im}(x)$.
Let $1 \leq j \leq n$ and let $G_j$ denote the $j$-th entry of $G$. 
We define the function $h(t):=G_j(tz + (1-t)w)$. The real and imaginary part of $h(t)$ are real differentiable functions of the real variable $t$.
The mean value theorem can be applied, and we find $0<t_1,t_2<1$ such that $\mathrm{Re}(h(1)) - \mathrm{Re}(h(0)) = \tfrac{\mathrm{d}}{\mathrm d t} \mathrm{Re}(h(t_1))$ and $\mathrm{Im}(h(1)) - \mathrm{Im}(h(0))  = \frac{\mathrm{d}}{\mathrm d t} \mathrm{Im}(h(t_2))$.
Setting $c_1 = t_1z+(1-t_1)w$ and $c_2 = t_2z+(1-t_2)w$ this implies
$$G_j(z) - G_j(w) = (\mathrm{Re}(G_j'(c_1)) +   i\mathrm{Im}(G_j'(c_2)))^T (z-w).$$ 
Similarly, we find $c_3,c_4$ on the line segment from $w$ to $x$ with 
$$G(w) - G(x) = (\mathrm{Re}(G_j'(c_3)) +   i\mathrm{Im}(G_j'(c_4))) (w-x).$$
We have $z-w = i(\mathrm{Im}(z-x))$ and $w-x = \mathrm{Re}(z-x)$. Therefore, 
\begin{equation}\label{eq24}
  \begin{aligned}
  \mathrm{Re}(G_j(z) - G_j(x)) &= -\mathrm{Im}(G_j'(c_2))^T\,\mathrm{Im}(z-x) + \mathrm{Re}(G_j'(c_3))^T\,\mathrm{Re}(z-x)\\
  & = \mathrm{Re}(a^T(z-x)),\\[0.5em]
  \mathrm{Im}(G_j(z) - G_j(x)) &= \mathrm{Re}(G_j'(c_1))^T\,\mathrm{Im}(z-x) + \mathrm{Im}(G_j'(c_4))^T \,\mathrm{Re}(z-x)\\
  &= \mathrm{Im}(b^T(z-x)),
  \end{aligned}
\end{equation}
where $a= \mathrm{Re}(G_j'(c_3)) + i \mathrm{Im}(G_j'(c_2))$ and $b=\mathrm{Re}(G_j'(c_1)) + i \mathrm{Im}(G_j'(c_4)).$ 
We get 
$$\vert G_j(z) - G_j(x)\vert  \leq \sqrt{2}\,\max\{ \vert a^T(z-x)\vert, \vert b^T(z-x)\vert\} \leq \sqrt{2}\,\max\{ \Vert a\Vert_1, \Vert b\Vert_1\}\, \Vert z-x\Vert_\infty,$$
where $\Vert \cdot \Vert_1$ is the $1$-norm (the dual of the $\infty$-norm). By construction, $c_1,c_2,c_3,c_4$ are contained in~$I$, because $I$ is convex. This implies that 
$a,b\in \mathrm{Re}(G_j'(I)) + i \mathrm{Im}(G_j'(I)).$ It follows, using \cref{eq11}, that 
$$\max\{ \Vert a\Vert_1, \Vert b\Vert_1\} \leq \Vert \mathbf 1_n - Y \cdot\square \mathrm{J}F(I)\Vert_\infty,$$
and so $\Vert G(z)-G(x)\Vert_\infty \leq \sqrt{2}\,\Vert \mathbf 1_n - Y \cdot\square \mathrm{J}F(I)\Vert_\infty\, \Vert z-x\Vert_\infty.$

For the second part, we have to show that for all $I\in\IC^n$ we have $G(I)\subset K_{x,Y}(I)$.
To show this we define the interval matrix $M:=(\mathbf 1_n - Y \square \mathrm{J}F(I))\in\IC^{n\times n}$.
By definition of $K_{x,Y}$ we have $G(x) + M(I-x) \subset K_{x,Y}(I)$.
Thus, we have to show that $G(z) - G(x) \in M(I-x)$, since $z\in I$ is arbitrary.
It follows from the equations in \cref{eq24} and using \cref{eq11} that we can find matrices $M_1,M_2\in M$ such that~$G(z) - G(x) = M_1 \mathrm{Re}(z-x) +i M_2 \mathrm{Im}(z-x)$.
Decomposing the matrices into real and imaginary part we find
$$G(z)-G(x) = \mathrm{Re}(M_1) \mathrm{Re}(z-x) - \mathrm{Im}(M_2) \mathrm{Im}(z-x) + i (\mathrm{Im}(M_1) \mathrm{Re}(z-x) + \mathrm{Re}(M_2) \mathrm{Im}(z-x)).$$
Since $z-x\in I$ and by definition of the complex interval multiplication from (\ref{complex_arithmetic}) and the interval matrix-vector-multiplication (\ref{def_matrix_mult}) we see that $G(z)-G(x)\in M (I-x)$.
This finishes the proof for the second part.

\end{proof}

\begin{proof}[Proof of Theorem \ref{thm:krawczyk} and Proposition \ref{interesting_prop}]
We fix $Y\in\mathbb C^{n\times n}$. The second part of Lemma \ref{lemma: oracle_ext} implies that, if we have $K_{x,Y} (I) \subseteq I$, then $G_Y(I) \subseteq I$. Brouwer's fixed point Theorem shows that $G_Y$ has a fixed point in~$I$. Since $Y$ is assumed to be invertible, the fixed point is a zero of $F$.
This finishes the proof for the first part of Theorem \ref{thm:krawczyk}.
For the second part let $z_0,z_1\in I$.
The first part of Lemma \ref{lemma: oracle_ext} implies
\begin{equation}\label{speed_of_convergence}
\lVert G_Y(z_0) - G_Y(z_1)  \rVert_\infty    \leq \sqrt{2} \lVert \mathbf 1_n- Y \cdot \square \mathrm{J}F(I) \rVert_\infty
 \lVert z_0 -z_1 \rVert_\infty.
\end{equation}
By assumption $ \sqrt{2} \lVert \mathbf 1_n- Y \cdot \square \mathrm{J}F(I) \rVert_\infty$ is smaller than 1 so $G_Y$ is a contraction.
Banach's Fixed Point Theorem implies that $G_Y$ has a unique fixed point in $I$. Since $Y$ is invertible, this fixed point is the unique zero of $F$ in $I$. This shows the second part of Theorem~\ref{thm:krawczyk}.

Finally, let $(x_i)_{i\geq 0}$ be the sequence defined in Proposition \ref{interesting_prop}. By assumption, $x_0\in I$ and $x_{i+1} = G_Y(x_{i})$, $i\geq 0$. The second part of Lemma \ref{lemma: oracle_ext} together with an induction argument imply that $x_i\in I$ for all $i$. Let $x^*$ be the unique zero of $F$ in $I$.
Then, \cref{speed_of_convergence} holds if we replace $z_0$ by $x^*$ and $z_1$ by $x_i$. We get $\lVert x^* - x_{i+1}  \rVert_\infty  < c\cdot \lVert x^* -x_i \rVert_\infty$ with a constant $c:=\sqrt{2} \lVert \mathbf 1_n- Y \cdot \square \mathrm{J}F(I) \rVert_\infty<1$. This shows linear convergence and finishes the proof of Proposition \ref{interesting_prop}.
\end{proof}

\section{Certifying reality}\label{sec:reality}
 For many applications only the real zeros of a polynomial system are of interest.
 Since numerical homotopy continuation computes in $\mathbb C^n$, it is important to have a rigorous method to determine whether a zero is real.

 Recall from Definition \ref{def_interval_zero} the notion of \emph{strong interval approximate zero}.

\begin{lemma}\label{lemma:real_solution}
Let $F: \C^n \rightarrow \C^n$ be a real square system of polynomials (or rational functions) and $I \in \IC^n$ a strong interval approximate zero of $F$. Then there exists $x\in I$ and $Y\in \C^{n\times n}$ satisfying $K_{x,Y}(I) \subset I$ and $\sqrt{2} \, \lVert \mathbf 1_n - Y \square \mathrm{J}F(I) \rVert_\infty < 1$.
  If additionally $\{ \bar{z} \, | \, z \in K_{x,Y}(I) \} \subset I$, the associated zero of $I$ is real.
\end{lemma}
\begin{proof}
  Theorem \ref{thm:krawczyk} implies that $F$ has a unique zero $s \in K_{x,Y}(I) \subset I$.
  The element wise complex conjugate $\bar{s}$ is also a zero of $F$.
  If we have that $\bar{s} \in \{ \bar{z} \, | \, z \in K_{x,Y}(I) \} \subset I$, then $\bar{s} = s$, since otherwise
  $\bar{s}$ and $s$ would be two distinct zeros of $F$ in~$I$, contradicting the uniqueness result from Theorem \ref{thm:krawczyk}.
\end{proof}

For a wide range of applications positive real zeros are of particular interest.
\begin{corollary}
  Let $F: \C^n \rightarrow \C^n$ be a real square system of polynomials and $I \in \IC^n$ a strong interval approximate zero of $F$ satisfying the conditions of Lemma \ref{lemma:real_solution}.
  If $\textnormal{Re}(I) > 0$ then the associated zero of $I$ is real and positive.
\end{corollary}

If the reality test in Lemma \ref{lemma:real_solution} fails for a strong interval approximate zero $I \in \C^n$ then this does not necessarily mean that the associated zero of $I$ is not real.
A sufficient condition that $I$ is not real is that there is a coordinate such that the imaginary part of it does not contain zero.
\begin{lemma}\label{lemma:complex_solution}
  Let $F(x)$ be a square system of polynomials or rational functions and let $I \in \IC^n$ be a strong interval approximate zero of $F$.
  If there exists $k \in \{1,\ldots,n\}$ such that $0 \notin \mathrm{Im}(I_k)$ then the associated zero of $I$ is not real.
\end{lemma}
\begin{proof}
  The associated zero $x$ of $I$ is contained in $I$. Since $0 \notin \mathrm{Im}(I_k)$ follows $x_k \notin \R$ and $x \notin \R^n$.
\end{proof}

Now assume that the certification routine produced a list $\mathcal{I}$ of $m$ distinct strong interval approximate zeros for a given system $F$, and that $m$ also agrees with the theoretical upper bound on the number of isolated, nonsingular zeros of $F$.
If we apply Lemma \ref{lemma:real_solution} to $I_k \in \mathcal{I}$, then we obtain only a
\emph{lower bound}, say $r$, on the number of real zeros of $F$.
However, combined with Lemma~\ref{lemma:complex_solution} we can also obtain an \emph{upper bound} of the number of real zeros.
If these two bounds agree we obtain a certificate that, among the associated zeros of the intervals in $\mathcal{I}$, there are \emph{exactly} $r$ real zeros.
An application of this is, e.g., the study of the distribution of the number of real solutions of the power flow equations \cite{Lindberg:Zachariah:Boston:Lesieutre:2020}.

 {\small
\bibliographystyle{alpha}
\bibliography{bibliography}
}
\end{document}